\documentclass{amsart}
  \usepackage{verbatim}
\usepackage{mathrsfs}
\usepackage{latexsym}
\usepackage{epsfig}
\usepackage{amsmath}
\usepackage{bbm}
\usepackage{graphics}
\usepackage{amssymb}
\usepackage{amsthm}
\usepackage{amsrefs}
\usepackage{amsopn}
\usepackage{amscd}
\usepackage[all,knot]{xy}
\xyoption{all}
\usepackage{rotating}
\usepackage{hyperref}

\numberwithin{equation}{section}

\theoremstyle{plain}
\newtheorem{thm}{Theorem}[section]
\newtheorem{lem}[thm]{Lemma}
\newtheorem{prop}[thm]{Proposition}
\newtheorem{cor}[thm]{Corollary}

\newtheorem*{thm*}{Theorem}
\newtheorem*{lem*}{Lemma}
\newtheorem*{prop*}{Proposition}
\newtheorem*{cor*}{Corollary}

\theoremstyle{definition}
\newtheorem{defn}[thm]{Definition}
\newtheorem*{defn*}{Definition}

{}
\newtheorem{rem}[thm]{Remark}
\newtheorem*{rem*}{Remark}

\newtheorem{notation}[thm]{Notation}{}
{}
{}
{}

\theoremstyle{remark}
{}
{}
{}

\def\cie{\subseteq}

\def\iso{\cong}

\def\un{\cup}

\def\intersec{\cap}

\def\to{\longrightarrow}

\def\rimp{\Rightarrow}

\def\int{\mathbb{Z}}

\def\str{\mathcal{O}}

\def\a{\alpha}

\def\s{\sigma}
\def\S{\Sigma}
\def\t{\tau}
\def\D{\mathsf{D}}
\def\SS{\mathsf{S}}

\DeclareMathOperator{\Spec}{Spec}
\DeclareMathOperator{\Sing}{Sing}

\DeclareMathOperator{\Spc}{Spc}

\DeclareMathOperator{\supp}{supp}

\DeclareMathOperator{\id}{id}

\DeclareMathOperator{\Hom}{Hom}

\DeclareMathOperator{\hocolim}{hocolim}

\DeclareMathOperator{\QCoh}{QCoh}

\DeclareMathOperator{\GInj}{GInj}
\DeclareMathOperator{\sGInj}{\underline{GInj}}

\title{Filtrations via tensor actions}
\author{Greg Stevenson}
\address{Universit\"at Bielefeld, Fakult\"at f\"ur Mathematik, BIREP Gruppe, Postfach 10\,01\,31, 33501 Bielefeld, Germany.}
\email{gstevens@math.uni-bielefeld.de}

\begin{document}



\begin{abstract}
\noindent We extend work of Balmer, associating filtrations of essentially small tensor triangulated categories to certain dimension functions, to the setting of actions of rigidly-compactly generated tensor triangulated categories on compactly generated triangulated categories. We show that the towers of triangles associated to such a filtration can be used to produce filtrations of Gorenstein injective quasi-coherent sheaves on Gorenstein schemes. This extends and gives a new proof of a result of Enochs and Huang. In the case of local complete intersections a further refinement of this filtration is given and we comment on some special properties of the associated spectral sequence in this case.
\end{abstract}

\maketitle

\tableofcontents

\section{Introduction}
In \cite{BaSpec} Paul Balmer has introduced tensor triangular geometry providing a unified framework in which to view the various classification theorems for thick tensor ideals of tensor triangulated categories in fields varying from algebraic geometry to KK-theory. Since then the study of the spectrum of a tensor triangulated category has taken on a life of its own; new examples have been computed \cite{DSgraded},\cite{DubeyMallick}, the abstract theory has been further developed (of course with applications) \cite{BaSSS},\cite{BaFilt}, and the framework has been extended to apply to a wider variety of examples \cite{BaRickard},\cite{StevensonActions}. The current work is concerned with extending some of the abstract theory developed by Balmer in \cite{BaFilt} to the context of tensor actions.

In \cite{BaFilt} Balmer has described how to produce a filtration of an essentially small tensor triangulated category given a dimension function on its spectrum. We prove that the analogous filtration exists for a compactly generated triangulated category $\mathcal{K}$ in the situation where a (nice) rigidly-compactly generated tensor triangulated $\mathcal{T}$ acts on $\mathcal{K}$ (in particular this covers the case where $\mathcal{T}$ acts on itself so extends these filtrations to the setup of \cite{BaRickard}). An advantage of working with these large triangulated categories is that the inclusions coming from the filtration admit coproduct preserving right adjoints. Thus one obtains for each object in $\mathcal{K}$ a functorial tower of triangles. This tower describes how the given object can be constructed from its local pieces as one knows is possible by the local-to-global principle. This is described in our main abstract result Theorem~\ref{thm_gen_filt}.

As an application we continue the study of the singularity category of a complete intersection. By applying the formalism we develop to the singularity category we obtain canonical filtrations of Gorenstein injective sheaves on Gorenstein schemes (Theorem~\ref{thm_gor_filt}). This extends work of Enochs and Huang \cite{EnochsHuang} and yields a new proof of their result in the affine case. In particular, it further strengthens the analogy between usual homological algebra and relative homological algebra. On a noetherian scheme $X$ one has, as in the case of commutative noetherian rings, a classification of indecomposable injective sheaves in terms of points of $X$. We show that, when $X$ is Gorenstein, any quasi-coherent Gorenstein injective sheaf on $X$ has a filtration whose factors decompose into pieces supported at points of $X$.

We also make some remarks concerning the spectral sequences associated to the filtrations we produce, both in general and in the special case of complete intersection rings where the $E_2$ page has some amusing special properties.

\section{Preliminary material}
We briefly recall some of the notation and constructions which we will use. For further details on tensor triangular geometry and tensor actions the interested reader should consult \cite{BaSpec} and \cite{StevensonActions} respectively (an excellent overview of tensor triangular geometry is also given in \cite{BaICM}).

We will work with \emph{rigidly-compactly generated tensor triangulated categories} throughout. Such a creature is a compactly generated tensor triangulated category $(\mathcal{T},\otimes,\mathbf{1})$ (as usual the monoidal structure is assumed to be symmetric, biexact, and preserve coproducts so that $\mathcal{T}$ has an internal hom, by Brown representability, which we denote by $\hom(-,-)$) such that $\mathcal{T}^{\mathrm{c}}$, the (essentially small) subcategory of compact objects, is a rigid tensor triangulated subcategory. Here $\mathcal{T}^c$ is a \emph{rigid tensor triangulated subcategory} if the monoidal structure and internal hom restrict to $\mathcal{T}^c$ (in particular the unit object $\mathbf{1}$ must be compact), and for all $a$ and $b$ in $\mathcal{T}^c$ the natural map
\begin{displaymath}
a^{\vee}\otimes b \to \hom(a,b)
\end{displaymath}
is an isomorphism, where, $a^{\vee} = \hom(a,\mathbf{1})$.

We associate to such a category its Balmer spectrum $\Spc \mathcal{T}^c$, as in \cite{BaSpec}, which is a spectral space universal with respect to parametrizing thick tensor ideals in terms of its Thomason subsets (i.e.\ those subsets which can be written as unions of closed subsets with quasi-compact complement). For a point $x\in \Spc \mathcal{T}^c$ we denote by $\mathcal{V}(x)$ its closure and by $\mathcal{Z}(x)$ those points which don't specialise to $x$ i.e.\
\begin{displaymath}
\mathcal{Z}(x) = \{y\in \Spc \mathcal{T}^c \; \vert \; x\notin \mathcal{V}(y)\}.
\end{displaymath}

Given a Thomason subset $\mathcal{V} \cie \Spc \mathcal{T}^c$ there are associated tensor idempotent objects of $\mathcal{T}$ denoted by $\mathit{\Gamma}_\mathcal{V}\mathbf{1}$ and $L_\mathcal{V}\mathbf{1}$ satisfying $\mathit{\Gamma}_\mathcal{V}\mathbf{1} \otimes L_\mathcal{V}\mathbf{1} = 0$. The construction of these idempotents can be found in \cite{BaRickard}. In the case that $\Spc \mathcal{T}^c$ is noetherian, which is the situation we will restrict ourselves to, one obtains a tensor idempotent for each point $x$ of the spectrum as follows
\begin{displaymath}
\mathit{\Gamma}_x\mathbf{1} = \mathit{\Gamma}_{\mathcal{V}(x)}\mathbf{1} \otimes L_{\mathcal{Z}(x)}\mathbf{1}.
\end{displaymath}

Now suppose we are given another compactly generated triangulated category $\mathcal{K}$. In \cite{StevensonActions}*{Definition~3.2} we have given a definition of a \emph{left action} of $\mathcal{T}$ on $\mathcal{K}$. This consists of a functor
\begin{displaymath}
\mathcal{T} \times \mathcal{K} \stackrel{\square}{\to} \mathcal{K}
\end{displaymath}
which is exact and coproduct preserving in each variable and satisfies certain natural compatibility conditions. We recall some details from the theory developed in \cite{StevensonActions} which will be required. We view $\mathcal{K}$ as a module for $\mathcal{T}$ and say that a localizing subcategory of $\mathcal{K}$ is a (localizing) \emph{submodule} if it is closed under the action of $\mathcal{T}$.

Let us assume henceforth that $\Spc \mathcal{T}^c$ is noetherian. The action of $\mathcal{T}$ on $\mathcal{K}$ gives rise to a notion of \emph{support} on $\mathcal{K}$ by setting, for $A\in \mathcal{K}$,
\begin{displaymath}
\supp A = \{x\in \Spc \mathcal{T}^c \; \vert \; \mathit{\Gamma}_xA \neq 0\},
\end{displaymath}
where $\mathit{\Gamma}_xA$ is shorthand for $\mathit{\Gamma}_x\mathbf{1}\square A$. The support gives rise to a pair of assignments
\begin{equation}\label{eq_st}
\left\{ \begin{array}{c}
\text{subsets of}\; \Spc\mathcal{T}^c 
\end{array} \right\}
\xymatrix{ \ar[r]<1ex>^\t \ar@{<-}[r]<-1ex>_\s &} \left\{
\begin{array}{c}
\text{localizing submodules} \; \text{of} \; \mathcal{K} \\
\end{array} \right\} 
\end{equation}
where, for a localizing submodule $\mathcal{L}$ we set
\begin{displaymath}
\s(\mathcal{L}) = \supp \mathcal{L} = \{x \in \Spc\mathcal{T}^c \; \vert \; \Gamma_x\mathcal{L} \neq 0\}
\end{displaymath}
and for a subset $W$ of $\Spc \mathcal{T}^c$ we set
\begin{displaymath}
\t(W) = \{A \in \mathcal{K} \; \vert \; \supp A \cie W\}.
\end{displaymath}

Finally, let us recall the local-to-global principle for an action.

\begin{defn}\label{defn_ltg}
We say $\mathcal{T}\times\mathcal{K} \stackrel{\square}{\to} \mathcal{K}$ satisfies the \emph{local-to-global principle} if for each $A$ in $\mathcal{K}$
\begin{displaymath}
\langle A \rangle_\square = \langle \Gamma_x A \; \vert \; x\in \Spc \mathcal{T}^c\rangle_\square,
\end{displaymath}
where for $\mathcal{A}\cie \mathcal{K}$ the notation $\langle \mathcal{A} \rangle_\square$ denotes the smallest localizing submodule containing the objects of $\mathcal{A}$.
\end{defn}

By \cite{StevensonActions}*{Theorem~6.9} the local-to-global principle holds whenever $\mathcal{T}$ is the homotopy category of a stable monoidal model category. Furthermore, it implies that for $A\in \mathcal{K}$ we have $\supp A = \varnothing$ if and only if $A\iso 0$.

\section{Discrete subsets and decompositions}\label{sec_disc}
Throughout this section $\mathcal{T}$ is a rigidly-compactly generated tensor triangulated category. We furthermore assume that $\mathcal{T}$ is the homotopy category of some stable monoidal model category and that $\Spc \mathcal{T}^c$ is noetherian. We fix a compactly generated triangulated category $\mathcal{K}$ and an action of $\mathcal{T}$ on $\mathcal{K}$. As noted above our hypotheses guarantee that the equivalent conditions of \cite{StevensonActions}*{Theorem~6.9} hold i.e., the action satisfies the local-to-global principle (\cite{StevensonActions}*{Definition~6.1}) and the associated support theory can distinguish the zero object of $\mathcal{K}$.

We begin by showing that the specialization relation for points of $\Spc \mathcal{T}^c$ controls the orthogonality relations between the categories $\mathit{\Gamma}_x\mathcal{K}$ for $x\in \Spc \mathcal{T}^c$.

\begin{lem}\label{lem_orthog}
Let $y$ be a point of $\Spc \mathcal{T}^c$ and let $\{x_i\}_{i\in I}$ be a collection of points of $\Spc \mathcal{T}^c$. Then for any object $A \in \mathcal{K}$ and any family $\{B_i\}_{i\in I}$ of objects of $\mathcal{K}$ 
\begin{displaymath}
\Hom(\mathit{\Gamma} _y A,\coprod_{i\in I} \mathit{\Gamma}_{x_i} B_i) \iso \Hom(\mathit{\Gamma}_y A, \coprod_{x_i\in \mathcal{V}(y)}\mathit{\Gamma}_{x_i} B_i).
\end{displaymath}
In particular, if none of the $x_i$ are specializations of $y$ we have
\begin{displaymath}
\coprod_{i\in I} \mathit{\Gamma}_{x_i} B_i \in \mathit{\Gamma}_y\mathcal{K}^{\perp}.
\end{displaymath}
\end{lem}
\begin{proof}
We claim there are isomorphisms
\begin{align*}
\Hom(\mathit{\Gamma}_y A, \coprod_{i\in I}\mathit{\Gamma}_{x_i} B_i) &\iso \Hom(\mathit{\Gamma}_y A, \mathit{\Gamma}_{\mathcal{V}(y)}\coprod_{i\in I}\mathit{\Gamma}_{x_i} B_i) \\
&\iso \Hom(\mathit{\Gamma}_y A, \coprod_{i\in I}\mathit{\Gamma}_{\mathcal{V}(y)} \mathit{\Gamma}_{x_i} B_i) \\
&\iso \Hom(\mathit{\Gamma}_x A, \coprod_{x_i\in \mathcal{V}(y)}\mathit{\Gamma}_{x_i} B_i). \\
\end{align*}
Using the definition of $\mathit{\Gamma}_y$ the first and second isomorphisms are just applications of the fact that $\mathit{\Gamma}_{\mathcal{V}(y)}$ is a coproduct preserving acyclization functor. The final isomorphism is a consequence of the fact that for any $i\in I$
\begin{displaymath}
\supp \mathit{\Gamma}_{\mathcal{V}(y)}\mathit{\Gamma}_{x_i} B_i = (\mathcal{V}(y)\intersec \{x_i\}),
\end{displaymath}
by \cite{StevensonActions}*{Proposition~5.7} and that this intersection is empty if and only if $\mathit{\Gamma}_{\mathcal{V}(y)}\mathit{\Gamma}_{x_i} B$ is zero by the local-to-global principle. Hence $\mathit{\Gamma}_{\mathcal{V}(y)}\mathit{\Gamma}_{x_i} B_i$ is non-zero if only if $x_i\in \mathcal{V}(y)$, and in this case $\mathcal{V}(x_i) \cie \mathcal{V}(y)$ so there is an isomorphism $\mathit{\Gamma}_{\mathcal{V}(y)}\mathit{\Gamma}_{x_i}B \iso \mathit{\Gamma}_{x_i}B$. 
\end{proof}

\begin{rem}\label{rem_totalorthog}
It follows that if $x$ and $y$ are points of $\Spc \mathcal{T}^c$ which are incomparable under the specialization ordering then $\mathit{\Gamma}_x\mathcal{K}$ and $\mathit{\Gamma}_y\mathcal{K}$ are mutually orthogonal i.e.,
\begin{displaymath}
\mathit{\Gamma}_x\mathcal{K} \cie \mathit{\Gamma}_y\mathcal{K}^\perp \quad \text{and} \quad \mathit{\Gamma}_y\mathcal{K} \cie \mathit{\Gamma}_x\mathcal{K}^\perp.
\end{displaymath}
\end{rem}

As in \cite{BIKStrat2} we introduce the following terminology:
\begin{defn}
We say a subset $W \cie \Spc \mathcal{T}^c$ is \emph{discrete} if for all distinct $x,y \in W$ we have
\begin{displaymath}
x\notin \mathcal{V}(y) \quad \text{and} \quad y\notin \mathcal{V}(x),
\end{displaymath}
i.e.,  there are no specialization relations between points of $W$.
\end{defn}

Given Remark \ref{rem_totalorthog} we are led to a decomposition of objects and categories with discrete supports (cf.\ \cite{BIKStrat2}*{Proposition~3.3}).

\begin{lem}\label{lem_discretemap}
Suppose $A$ is an object of $\mathcal{K}$ such that $\supp A$ is discrete. Then for every $x\in \Spc \mathcal{T}^c$ there is a natural map
\begin{displaymath}
\mathit{\Gamma}_x A \to A.
\end{displaymath}
\end{lem}
\begin{proof}
We may as well assume $x\in \supp A$ as there is little difficulty in producing a natural map from the zero object to $A$.

There is a natural map
\begin{displaymath}
\mathit{\Gamma}_{\mathcal{V}(x)}A \to A.
\end{displaymath}
As $\supp A$ is discrete we have
\begin{displaymath}
\supp \mathit{\Gamma}_{\mathcal{V}(x)}A = \supp A \intersec \mathcal{V}(x) = \{x\}
\end{displaymath}
(using \cite{StevensonActions}*{Proposition~5.7}) so $\mathit{\Gamma}_{\mathcal{V}(x)}A$ lies in $L_{\mathcal{Z}(x)}\mathcal{K}$. In particular we have natural isomorphisms
\begin{displaymath}
\mathit{\Gamma}_{\mathcal{V}(x)}A \iso L_{\mathcal{Z}(x)}\mathit{\Gamma}_{\mathcal{V}(x)}A \iso \mathit{\Gamma}_x A
\end{displaymath}
from which we deduce the desired natural morphism.
\end{proof}

\begin{rem}\label{rem_lulz}
Let $A$ be as in the lemma. It is clear by the construction of the morphism $\mathit{\Gamma}_x A \to A$ that its image under $\mathit{\Gamma}_x$ is an isomorphism.
\end{rem}

Before stating the decomposition we introduce some notation.

\begin{notation}
Let $\mathcal{L}_i$ be a collection of localizing subcategories of $\mathcal{K}$ indexed by a set $I$. We denote by $\prod_{i\in I} \mathcal{L}_i$ 
the subcategory whose objects are coproducts of objects of the $\mathcal{L}_i$ with componentwise morphisms and triangulated structure (here componentwise means with respect to $I$).
\end{notation}

\begin{prop}\label{prop_discrete_decomp}
Let $W$ be a discrete subset of $\Spc \mathcal{T}^c$. Then the localizing subcategory $\t(W)$ of $\mathcal{K}$ decomposes as
\begin{displaymath}
\prod_{x\in W} \mathit{\Gamma}_x\mathcal{K}
\end{displaymath}
in a way which is compatible with the triangulated structure.
\end{prop}
\begin{proof}
The result is trivial if $W$ consists of fewer than two points, so we may as well assume $W$ is neither empty nor a singleton.

We first show that every object of the localizing submodule $\t(W)$ decomposes in the required fashion. Suppose $A$ is an object of $\mathcal{K}$ with $\supp A \cie W$. Then we have, by Lemma \ref{lem_discretemap}, a natural morphism
\begin{displaymath}
\coprod_{x\in W} \mathit{\Gamma}_x A \to A
\end{displaymath}
whose mapping cone we shall denote by $Z$. It is easily seen, by applying $\mathit{\Gamma}_y$ for any $y\in \Spc \mathcal{T}^c$ to the associated triangle and using the observation in Remark \ref{rem_lulz}, that $\supp Z = \varnothing$. By \cite{StevensonActions}*{Theorem~6.9} the object $Z$ must then be zero and so the natural map from the coproduct is an isomorphism.

To finish the proof it just remains to note that $\t(W)$ is, by definition, triangulated, generated by the $\mathit{\Gamma}_x\mathcal{K}$ for $x\in W$ (by the local-to-global principle), and for any $A,B\in \t(W)$ we have
\begin{align*}
\Hom(A,B) &\iso \Hom(\coprod_{x\in W}\mathit{\Gamma}_x A, \coprod_{y\in W}\mathit{\Gamma}_y B) \\
&\iso \prod_{x\in W} \Hom(\mathit{\Gamma}_x A, \mathit{\Gamma}_x B)
\end{align*}
by Lemma \ref{lem_orthog} and discreteness of $W$. Thus not only does every object of $\t(W)$ decompose in the required way but the non-zero morphisms only occur between the components corresponding to identical points. Hence the triangulated structure inherited from $\mathcal{K}$ is the componentwise one.
\end{proof}

\section{Filtrations by supports}\label{sec_filt}
We continue the notation of the previous section: $\mathcal{T}$ is a rigidly-compactly generated tensor triangulated category with a monoidal model and noetherian spectrum, and $\mathcal{K}$ is a compactly generated triangulated category on which $\mathcal{T}$ acts.

Let us first recall the notion of a dimension function as in \cite{BaFilt}*{Definition~3.1}. Let $\int\un\{\pm \infty\}$ be the ordered set where we just extend the usual order on $\int$ by $-\infty < d < \infty$ for all $d\in \int$.

\begin{defn}\label{defn_dimf}
A \emph{dimension function} on $\mathcal{T}$ is a function $\dim\colon \Spc\mathcal{T}^c \to \int\un\{\pm \infty\}$ satisfying:
\begin{itemize}
\item[(1)] if $x\in \mathcal{V}(y)$ then $\dim x \leq \dim y$;
\item[(2)] if $x\in \mathcal{V}(y)$ and $\dim x = \dim y$ is finite then $x=y$.
\end{itemize}
\end{defn}

Given a dimension function and $d\in \int\un\{\pm \infty\}$ we set
\begin{displaymath}
\mathcal{H}_{\leq d} = \{x\in \Spc \mathcal{T}^c \; \vert \; \dim x \leq d\}.
\end{displaymath}
By condition (1) above $\mathcal{H}_{\leq d}$ is specialization closed. We thus have associated acyclization and localization functors on $\mathcal{K}$, via the action of $\mathcal{T}$, we which denote by $\mathit{\Gamma}_{\leq d}$ and $L_{\leq d}$ respectively. We also wish to consider the subsets
\begin{displaymath}
\mathcal{H}_d = \{x\in \Spc\mathcal{T}^c \; \vert \; \dim x = d\}.
\end{displaymath}
Let us begin with some simple observations. For $d'\leq d$ in $\int\un\{\pm \infty\}$ there is a containment of subsets $\mathcal{H}_{\leq d'} \cie \mathcal{H}_{\leq d}$. This gives rise to a corresponding inclusion of the associated smashing subcategories and hence a natural transformation
\begin{displaymath}
\mathit{\Gamma}_{\leq d'} \to \mathit{\Gamma}_{\leq d}.
\end{displaymath}
There is an evident filtration of $\mathcal{K}$ by these smashing subcategories
\begin{displaymath}
0 \cie \mathit{\Gamma}_{\leq -\infty}\mathcal{K} \cie \cdots \cie \mathit{\Gamma}_{\leq d-1}\mathcal{K} \cie \mathit{\Gamma}_{\leq d}\mathcal{K} \cie \cdots \cie \mathit{\Gamma}_{\leq \infty}\mathcal{K} = \mathcal{K}.
\end{displaymath}
Let us now consider the coproduct preserving functor $\mathit{\Gamma}_d = \mathit{\Gamma}_{\leq d} L_{\leq {d-1}}$ (note that $\mathit{\Gamma}_{\leq d}$ and $L_{\leq {d-1}}$ commute up to natural isomorphism). For any object $A$ of $\mathcal{K}$ we have a localization triangle
\begin{displaymath}
\mathit{\Gamma}_{\leq d-1}A \to \mathit{\Gamma}_{\leq d}A \to \mathit{\Gamma}_dA \to \S\mathit{\Gamma}_{\leq d-1}A.
\end{displaymath}
The functor $\mathit{\Gamma}_d$ realises the localizing submodule $\t(\mathcal{H}_d)$, where $\t$ is as in \ref{eq_st}.

\begin{lem}
The essential image of $\mathit{\Gamma}_d$, denoted $\mathit{\Gamma}_d\mathcal{K}$, is precisely the localizing submodule of objects whose support lies in $\mathcal{H}_d$ i.e.\ $\t(\mathcal{H}_d) = \mathit{\Gamma}_{d}\mathcal{K}$.
\end{lem}
\begin{proof}
It follows immediately from the properties of the support, in particular \cite{StevensonActions}*{Proposition~5.7~(4)}, that $\mathit{\Gamma}_d\mathcal{K} \cie \t(\mathcal{H}_d)$. On the other hand suppose $A\in \mathcal{K}$ is an object of $\t(\mathcal{H}_d)$ i.e., the support of $A$ is contained in $\mathcal{H}_d$. Consider the localization triangle
\begin{displaymath}
\mathit{\Gamma}_{\leq d}A \to A \to L_{\leq d}A \to \S \mathit{\Gamma}_{\leq d}.
\end{displaymath}
Again using \cite{StevensonActions}*{Proposition~5.7~(4)} we deduce that $\supp L_{\leq d}A$ is empty and so $L_{\leq d}A$ must be zero. Hence $\mathit{\Gamma}_{\leq d}A \iso A$. The analogous argument for $L_{\leq d-1}$ shows that $L_{\leq d-1} A \iso A$ and so we deduce that $\mathit{\Gamma}_d A \iso A$ as desired.
\end{proof}

Hence we obtain a tower of triangles for $A$ relative to the given dimension function
\begin{displaymath}
\xymatrix{
\cdots \ar[r] &\mathit{\Gamma}_{\leq d-2}A \ar[r] & \mathit{\Gamma}_{\leq d-1}A \ar[r] \ar[d] & \mathit{\Gamma}_{\leq d}A \ar[r] \ar[d] & \cdots &\\
&& A_{d-1} \ar[ul]^{\S} & A_{d} \ar[ul]^{\S}
}
\end{displaymath}
whose $d$th factor $A_d = \mathit{\Gamma}_dA$ lies in $\mathit{\Gamma}_d\mathcal{K} = \t(\mathcal{H}_d)$. Furthermore, this tower is functorial in $A$.

The factors $A_d$ admit a decomposition into local pieces.

\begin{lem}\label{lem_disc_decomp}
For $d\in \int$ the subset $\mathcal{H}_d$ is discrete, so there are equalities of subcategories
\begin{displaymath}
\mathit{\Gamma}_d\mathcal{K} = \t(\mathcal{H}_d) = \prod_{x\in \mathcal{H}_d} \mathit{\Gamma}_x \mathcal{K}.
\end{displaymath}
\end{lem}
\begin{proof}
That $\mathcal{H}_d$ is discrete is immediate from (2) of Definition \ref{defn_dimf}. The second statement then follows from Proposition \ref{prop_discrete_decomp}.
\end{proof}

The following theorem summarises the observations we have thus far made.

\begin{thm}\label{thm_gen_filt}
Suppose $\mathcal{T}$ is a rigidly-compactly generated tensor triangulated category whose compacts have noetherian spectrum and which can be presented as the homotopy category of a stable monoidal model category. Let $\mathcal{K}$ be a compactly generated triangulated category on which $\mathcal{T}$ acts. Then given a dimension function $\dim$ on $\Spc\mathcal{T}^c$ the functors $\mathit{\Gamma}_{\leq d}$ for $d\in \int\cup\{\pm \infty\}$, associated to the dimension function, yield a filtration of $\mathcal{K}$ by localizing submodules
\begin{displaymath}
\cdots \cie \mathit{\Gamma}_{\leq d}\mathcal{K} \cie \mathit{\Gamma}_{\leq d+1}\mathcal{K} \cie \cdots \cie \mathit{\Gamma}_{\leq \infty}\mathcal{K} = \mathcal{K}.
\end{displaymath}
For $d\in \int$ the factors in this filtration are, up to equivalence, the subcategories
\begin{displaymath}
\prod_{x\in \mathcal{H}_d}\mathit{\Gamma}_x\mathcal{K}.
\end{displaymath}
Thus each $A$ in $\mathcal{K}$ has a unique (up to isomorphism) functorial tower
\begin{displaymath}
\xymatrix{
\cdots \ar[r] &\mathit{\Gamma}_{\leq d-2}A \ar[r] & \mathit{\Gamma}_{\leq d-1}A \ar[r] \ar[d] & \mathit{\Gamma}_{\leq d}A \ar[r] \ar[d] & \cdots &\\
&& A_{d-1} \ar[ul]^{\S} & A_{d} \ar[ul]^{\S}
}
\end{displaymath}
such that, for $d\in \int$
\begin{displaymath}
A_d \iso \coprod_{x\in \mathcal{H}_d} \mathit{\Gamma}_xA.
\end{displaymath}
\end{thm}

This has the pleasant consequence of describing, in some sense, how to assemble an object from local pieces. Our hypotheses are such that the local-to-global principle holds so one knows that, for $A\in \mathcal{K}$, the localizing submodule generated by the $\mathit{\Gamma}_xA$ for $x\in \Spc\mathcal{T}^c$ contains $A$. Given a dimension function, such that the filtration on $\mathcal{K}$ is bounded below, the tower of the theorem provides us with a sequence of canonical triangles expressing a way of building $A$. In fact it evidently gives the following slight sharpening of the local-to-global principle (cf.\ \cite{BIKStrat2}*{Theorem~3.4}).

\begin{cor}
Suppose $\dim$ is a dimension function on $\Spc \mathcal{T}^c$ and $A$ is an object of $\mathcal{K}$. If there exists a $d_0\in \int$ such that $\mathit{\Gamma}_{\leq d_0}A = 0$ and no $x\in \Spc \mathcal{T}^c$ satisfies $\dim(x) = \infty$ then
\begin{displaymath}
A \in \langle \mathit{\Gamma}_xA \; \vert \; x\in \Spc \mathcal{T}^c \rangle_\mathrm{loc}.
\end{displaymath}
The conclusion also holds if there exist $d_0, d_1\in \int$ with $\mathit{\Gamma}_{\leq d_0}A = 0$ and $\mathit{\Gamma}_{\leq d_1}A \iso A$.
\end{cor}
\begin{proof}
If there exist $d_0$ and $d_1$ as in the second statement then there is a $d_-$ such that $\mathit{\Gamma}_{\leq d_-}A = \mathit{\Gamma}_{d_-}A$. As each $A_d$ lies in $\langle \mathit{\Gamma}_xA \; \vert \; x\in \Spc \mathcal{T}^c \rangle_\mathrm{loc}$ it follows that we can build $\mathit{\Gamma}_{\leq d}A$ for all $d\in \int$. This proves the second statement of the corollary.

To complete the proof of the first statement we note that, as every point of $\Spc \mathcal{T}^c$ has finite dimension, we have $\Spc \mathcal{T}^c = \cup_{d<\infty} \mathcal{H}_d$. Thus it follows from \cite{StevensonActions}*{Lemma~6.6} that $A \iso \hocolim_{d< \infty} \mathit{\Gamma}_{\leq d} A$. So, combining this with what we have already shown, $A$ is a homotopy colimit of objects of $\langle \mathit{\Gamma}_xA \; \vert \; x\in \Spc \mathcal{T}^c \rangle_\mathrm{loc}$ and hence also lies in this localizing subcategory.
\end{proof}

\section{Filtrations of singularity categories}\label{ssec_filtrations}
In this section we give a first example of the machinery developed above by applying it to singularity categories. Let $X$ be a noetherian separated scheme and denote by $\SS(X)$ the homotopy category of acyclic complexes of injective quasi-coherent $\str_X$-modules as in \cite{KrStab}. This category carries an action of $\D(X)$, the derived category of quasi-coherent $\str_X$-modules, in the sense of \cite{StevensonActions}: there is a functor
\begin{displaymath}
\D(X) \times \SS(X) \stackrel{\odot}{\to} \SS(X)
\end{displaymath}
which is exact and coproduct preserving in each variable as well as satisfying compatibility conditions which allow one to view $\SS(X)$ as a $\D(X)$-module (details can be found in \cite{Stevensonclass}*{Section~3}). This action gives rise to supports valued in $\Sing X = \s \SS(X)$ the singular locus of $X$ i.e., those points $x\in X$ such that $\str_{X,x}$ is not regular (see \cite{Stevensonclass}*{Lemma~7.4} for the computation of $\s \SS(X)$).


The dimension function on $X = \Spc\D(X)^c$ we will be concerned with is given by $\dim(x) = \dim_\mathrm{Krull}(\mathcal{V}(x))$ the Krull dimension of the closure of $x$. From Theorem \ref{thm_gen_filt} we obtain an associated filtration of $\SS(X)$.

\begin{thm}\label{thm_filt}
The functors $\mathit{\Gamma}_{\leq d}$ for $0\leq d \leq \dim X$, associated to the Krull dimension function, endow $\SS(X)$ with a filtration by localizing submodules
\begin{displaymath}
0 = \mathit{\Gamma}_{\leq -1}\mathcal{K} \cie \mathit{\Gamma}_{\leq 0}\mathcal{K} \cie \mathit{\Gamma}_{\leq d}\SS(X) \cie \mathit{\Gamma}_{\leq d+1}\SS(X) \cie \cdots \cie \mathit{\Gamma}_{\leq \infty}\SS(X) = \SS(X)
\end{displaymath}
whose factors are up to equivalence the subcategories
\begin{displaymath}
\prod_{x\in \mathcal{H}_d}\mathit{\Gamma}_x\SS(X)
\end{displaymath}
In particular, if $\Sing X$ has finite Krull dimension we get a finite filtration
\begin{displaymath}
0  \cie \mathit{\Gamma}_{\leq 0}\SS(X) \cie \mathit{\Gamma}_{\leq 1}\SS(X) \cie \cdots \cie \mathit{\Gamma}_{\leq n-1}\SS(X) \cie \mathit{\Gamma}_{\leq n}\SS(X) = \SS(X)
\end{displaymath}
so that each $A$ in $\SS(X)$ is equipped with a unique (up to isomorphism), finite, and functorial tower
\begin{displaymath}
\xymatrix{
0 \ar[r] & \mathit{\Gamma}_{\leq 0}A \ar[r] \ar[d] & \mathit{\Gamma}_{\leq 1}A \ar[r] \ar[d] & \cdots \ar[r] & \mathit{\Gamma}_{\leq n-1}A \ar[r]  & \mathit{\Gamma}_{\leq n}A = A \ar[d] \\
& A_0 \ar[ul]^{\S} & A_{1} \ar[ul]^{\S} &&& A_n \ar[ul]^{\S}
}
\end{displaymath}
such that the $A_d$ live in
\begin{displaymath}
\mathit{\Gamma}_d\SS(X) = \prod_{x\in \mathcal{H}_d} \mathit{\Gamma}_x \SS(X).
\end{displaymath}
\end{thm}

This can be viewed as a non-affine stable version of the filtration of \cite{EnochsHuang} when $X$ is Gorenstein. Indeed, in this case we can identify $\SS(X)$ with $\sGInj X$ the stable category of Gorenstein injective quasi-coherent $\str_X$-modules by \cite{KrStab}*{Proposition 7.13}, and so we have a filtration up to injective summands. We will now prove that these towers correspond to actual filtrations of quasi-coherent Gorenstein injective $\str_X$-modules; this gives a new proof of \cite{EnochsHuang}*{Theorem~3.1} when $X$ is affine. 
We now furthermore assume that $X$ is Gorenstein of Krull dimension $n$. Let us recall what it means for a quasi-coherent sheaf on a Gorenstein scheme to be Gorenstein injective (further details on Gorenstein injective modules may be found, for instance, in \cite{EnochsJenda}).

\begin{defn}\label{defn_ginj}
A quasi-coherent $\str_X$-module $G$ is \emph{Gorenstein injective} if there exists an exact sequence
\begin{displaymath}
E = \cdots \to E_1 \to E_0 \to E^0 \to E^1 \to \cdots
\end{displaymath}
of injective quasi-coherent $\str_X$-modules where $G = \ker( E^0 \to E^1)$ is the zeroth syzygy of $E$. We call $E$ a \emph{complete injective resolution} of $G$. We denote the full subcategory of Gorenstein injective sheaves by $\GInj X$.
\end{defn}

\begin{rem}
This is not the ``usual'' definition of a Gorenstein injective sheaf. One normally further requires that the complex of abelian groups $\Hom(I,E)$ is exact for every injective quasi-coherent sheaf $I$. However, this condition is automatically satisfied by any acyclic complex of injectives when $X$ is Gorenstein; one can deduce this either from \cite{KrStab}*{Proposition~7.13} or \cite{MurfetTAC}.
\end{rem}

The category $\GInj X$ inherits the structure of an exact category from $\QCoh X$. Moreover, it is a Frobenius category and thus the stable category $\sGInj X$, obtained by annihilating the projective/injective objects, is triangulated. As noted above it is equivalent to $\SS(X)$ as the scheme $X$ is Gorenstein. For a quasi-coherent Gorenstein injective sheaf $G$ we use $\pi(G)$ to denote its image in the stable category.

We first show Gorenstein injective quasi-coherent sheaves are particularly well behaved with respect to the local cohomology functors on $\D(X)$ (a fact which is presumably well known). This is a rather trivial consequence of the following lemma which, in analogy with \cite{BIK}*{Theorem~9.1}, justifies calling the $\mathit{\Gamma}_\mathcal{V}$ local cohomology functors. We denote by $\mathbf{R}\underline{H}_\mathcal{V}(-)$ the right derived functor of $\underline{H}_\mathcal{V}(-)$, the subsheaf with supports in $\mathcal{V}$.

\begin{lem}\label{lem_loccohom}
Let $\mathcal{V}$ be a specialization closed subset of $X$. There is a natural isomorphism of endofunctors on $\D(X)$
\begin{displaymath}
\mathit{\Gamma}_{\mathcal{V}} \iso \mathbf{R}\underline{H}_\mathcal{V}(-).
\end{displaymath}
\end{lem}
\begin{proof}
Since the support of  $\mathbf{R}\underline{H}_\mathcal{V}(-)$ is, essentially by definition, contained in $\mathcal{V}$ there is a canonical natural transformation $\mathbf{R}\underline{H}_\mathcal{V}(-) \to \mathit{\Gamma}_\mathcal{V}$. By \cite{StevensonActions}*{Lemma~8.1} the functor $\mathit{\Gamma}_\mathcal{V}$ restricted to an open affine $U$ is precisely $\mathit{\Gamma}_{\mathcal{V}\intersec U}$ in $D(U)$. Thus, since on an affine scheme the canonical comparison map is an isomorphism by \cite{BIK}*{Theorem~9.1} (on an affine scheme our $\mathit{\Gamma}_\mathcal{V}$ agrees with the construction of Benson, Iyengar, and Krause by \cite{StevensonActions}*{Proposition~9.4}), the map $\mathbf{R}\underline{H}_\mathcal{V}(-) \to \mathit{\Gamma}_\mathcal{V}$ is an isomorphism as claimed as we can check this on an open affine cover.
\end{proof}

\begin{prop}\label{prop_ginj_acyclicity}
Let $G$ be a Gorenstein injective sheaf considered as an object of $\D(X)$ concentrated in degree zero. Then $\mathit{\Gamma}_\mathcal{V}G$ is again, up to isomorphism, a Gorenstein injective sheaf concentrated in degree zero.
\end{prop}
\begin{proof}
Let $E$ be a complete injective resolution of $G$. The brutal truncation $E^{\geq 0}$ is an injective resolution for $G$. Thus we may compute $\mathit{\Gamma}_\mathcal{V}G \iso \mathbf{R}\underline{H}_\mathcal{V}(G)$ by applying $\underline{H}_\mathcal{V}$ to $E$ and truncating.

Observe that $\underline{H}_\mathcal{V}(E)$ is again a complex of injectives since $\underline{H}_\mathcal{V}$ either kills or preserves any indecomposable injective. Furthermore it is an acyclic complex since $\mathbf{R}^i\underline{H}_\mathcal{V}$ vanishes on all quasi-coherent sheaves for $i\geq \dim X = n$ (see for example \cite{HartshorneLC}*{Proposition~1.12}) and any truncation of $E$ is the resolution of a sheaf.

Thus the complex $\underline{H}_\mathcal{V}(E)$ is an acyclic complex of injectives with zeroth syzygy $\underline{H}_\mathcal{V}(G)$ which proves that $\mathit{\Gamma}_\mathcal{V}G$ is, up to isomorphism, a Gorenstein injective sheaf concentrated in degree zero.
\end{proof}

\begin{rem}
This result also gives us information about the action of $\D(X)$ on $\SS(X) \iso \sGInj X$. As above let us denote the projection $\GInj X \to \sGInj X$ by $\pi$. Combining the proposition with \cite{Stevensonclass}*{Proposition~5.3} and Lemma~5.4 gives isomorphisms in $\sGInj X$
\begin{displaymath}
\mathit{\Gamma}_{\leq i}(\pi G) \iso \pi(\mathit{\Gamma}_{\leq i}G) \iso \pi(\underline{H}_{\mathcal{H}_{\leq i}}(G))
\end{displaymath}
for all $G\in \GInj X$.
\end{rem}

\begin{thm}\label{thm_gor_filt}
Suppose $X$ is a separated Gorenstein scheme of Krull dimension $n$ and let $G$ be a quasi-coherent Gorenstein injective $\str_X$-module. Then $G$ has a filtration by Gorenstein injective subsheaves
\begin{displaymath}
0 = G_{\leq -1} \cie G_{\leq 0} \cie \cdots \cie G_{\leq n-1} \cie G_{\leq n} = G
\end{displaymath}
where each factor $G_{\leq i}/G_{\leq i-1}$ is Gorenstein injective and has a direct sum decomposition indexed by $\mathcal{H}_i$, the irreducible closed subsets of $X$ of dimension $i$. The factor $G_x$, corresponding to a point $x$, is $\mathfrak{m}_x$-local and $\mathfrak{m}_x$-torsion where $\mathfrak{m}_x\cie \str_{X,x}$ is the maximal ideal of the stalk of the structure sheaf at $x$. If $x\notin \Sing X$ then the summand corresponding to $x$ is either injective or zero. These filtrations are unique up to isomorphism and functorial. Furthermore, the image of this filtration in the stable category is the tower of Theorem \ref{thm_filt}.
\end{thm}
\begin{proof}
Let $G$ be a Gorenstein injective sheaf and let us denote $\mathit{\Gamma}_{{\leq i}}G$ by $G_{\leq i}$ for $i\in \int$. The objects $G_{\leq i} \iso \underline{H}_{\mathcal{H}_{\leq i}}(G)$ are again Gorenstein injective sheaves in degree zero by Proposition \ref{prop_ginj_acyclicity}. For $i\leq j$ the canonical map $G_{\leq i} \to G_{\leq j}$ in the localization triangle
\begin{displaymath}
G_{\leq i} \to G_{\leq j} \to L_{\leq i} G_{\leq j} \to \S G_{\leq i}
\end{displaymath}
in $\D(X)$ induces a monomorphism on cohomology, namely it gives the canonical inclusion $\underline{H}_{\mathcal{H}_{\leq i}}(G) \cie \underline{H}_{\mathcal{H}_{\leq j}}(G)$. Combining these observations we see this triangle corresponds to an exact sequence
\begin{displaymath}
0 \to G_{\leq i} \to G_{\leq j} \to L_{\leq i}G_{\leq j} \to 0
\end{displaymath}
of quasi-coherent sheaves. By (the sheaf analogue of) \cite{EnochsJenda}*{Theorem~10.1.4} the $\str_X$-module $L_{\leq i}G_{\leq j}$ is also Gorenstein injective as it is the cokernel of a monomorphism between Gorenstein injectives. This produces the filtration and shows that its avatar in the derived category is the tower associated to our dimension function for the action of $\D(X)$ on itself. Projecting to the stable category gives the tower of Theorem \ref{thm_filt}. It is clear by construction that the filtration is functorial and unique up to isomorphism.

The factors in the filtration decompose as claimed essentially by Lemma \ref{lem_disc_decomp}. We have seen above that all of the triangles in the tower for $G$ in $\D(X)$ are given by short exact sequences in $\GInj X$. Since the $G_i$ decompose in $\D(X)$ they must already decompose into a sum of the desired form in $\GInj X$. That the summands in this decomposition have the claimed properties, i.e.\ are torsion and local with respect to the appropriate maximal ideal, is evident.


Finally, the assertion that, for each factor, summands corresponding to points not in the singular locus are either injective or zero is immediate from the fact that the support of $\SS(X)$ is precisely $\Sing X$ \cite{Stevensonclass}*{Lemma~7.4}; if $x\notin \Sing X$ then $\pi(\mathit{\Gamma}_x G) \iso \mathit{\Gamma}_x \pi(G) = 0$ implying $\mathit{\Gamma}_x G$ must be injective.
\end{proof}

\section{A refinement for complete intersections}

We now restrict ourselves to considering a (non-abstract) local complete intersection $(R,\mathfrak{m},k)$ and point out that one can, up to an injective, refine this filtration. Let us suppose that $R$ has codimension $c$ i.e.\
\begin{displaymath}
\dim_k \mathfrak{m}/\mathfrak{m}^2 - \dim R = c.
\end{displaymath}
Equivalently the complexity of the residue field is $c$.

Let $Y$ be a generic hypersurface for $R$ as in \cite{OrlovSing2}*{Section~2}. We recall the construction of such a hypersurface. Choose a regular local ring $Q$ and a surjection $Q\to R$ with kernel generated by a regular sequence of length $c$. Set $\mathcal{E} = \str_{\Spec Q}^{\oplus c}$, and consider the section $(q_1,\ldots,q_c)$ of $\mathcal{E}$ where the $q_i$ form a regular sequence generating the kernel of $Q\to R$. The hypersurface $Y$ is defined by the corresponding section $\Sigma_{i=1}^c q_i x_i$ of $\str_{\mathbb{P}^{c-1}_Q}(1)$ where the $x_i$ form a basis for the free $Q$-module $H^0(\mathbb{P}^{c-1}_Q, \str_{\mathbb{P}^{c-1}_Q}(1))$. Our setup is summarised by the commutative diagram
\begin{equation}\label{eq_1}
\xymatrix{
\mathbb{P}^{c-1}_R \ar[r]^-i \ar[d]_p & Y \ar[r]^-u & \mathbb{P}^{c-1}_Q \ar[d]^{q} \\
\Spec R \ar[rr]_-j && \Spec Q.
}
\end{equation}

It is proved in \cite{Stevensonclass}*{Corollary~10.5} that $\Sing Y$ controls the localizing subcategories of $S(R)$. Our aim is to show that this classification of localizing subcategories in terms of subsets of $\Sing Y$ is reflected by a refinement of the filtrations of Theorem~\ref{thm_gor_filt}.

\subsection{Local periodicity}\label{ssec_periodicity}
We begin with a brief aside concerning the local structure of $\SS(R)$ for a local complete intersection $R$ as above. Since we will use this periodicity in our refined filtration, and it was not made explicit in \cite{Stevensonclass}, we include the details here.

Let us denote by $\mathcal{L}$ the line bundle $u^*\str(1)_{\mathbb{P}^{c-1}_Q}$ on $Y$ and by $\{t_i\}_{i=1}^c$ global sections of $\mathcal{L}$ corresponding to the embedding $u$. The sections $t_i$ give a standard open affine cover of $Y$ by the open subschemes
\begin{displaymath}
Y_{t_i} \iso \Spec Q[X_1,\ldots, \hat{X}_i,\ldots, X_{c}]/(\sum_{j\neq i} q_jX_j + q_i),
\end{displaymath}
where as usual $\hat{X}_i$ means the $i$th indeterminate is omitted

Each $Y_{t_i}$ is an affine hypersurface and so the suspension in $\SS(Y_{t_i})$ is 2-periodic. We can view this through the lens of \cite{Stevensonclass}*{Lemma~7.8} which identifies the action of $\mathcal{L}$ on $\SS(Y)$ with $\S^2$. The action is compatible with restriction to open affines by \cite{Stevensonclass}*{Lemma~7.3} so we deduce from the morphism
\begin{displaymath}
\id_{\SS(Y)} \iso \str_Y \odot (-) \to \mathcal{L} \odot (-) \iso \S^2,
\end{displaymath}
corresponding to the section $t_i$, a natural isomorphism $\id_{\SS(Y_{t_i})} \stackrel{\sim}{\to} \S^2$ of endofunctors on $\SS(Y_{t_i})$. Essentially this cover together with the compatibility conditions of \cite{StevensonActions}*{Section~8} gives, by the above reasoning, the following result.

\begin{prop}\label{prop_periodic}
Suppose $A$ is an object of $\SS(R)$ such that
\begin{displaymath}
\supp_{D(Y)} A \cie Y_{t_i}
\end{displaymath}
for some $1\leq i \leq c$. Then $t_i\colon \str_Y \to \mathcal{L}$ gives a natural isomorphism $A \stackrel{\sim}{\to} \S^2 A$. In particular for every point $y\in \Sing Y$ the subcategory $\mathit{\Gamma}_y \SS(R)$ is $2$-periodic.
\end{prop}

As a corollary we obtain a geometric proof of a special case of a result due to Jesse Burke \cite{BurkeID}.

\begin{cor}
Suppose $A$ is an object of $\SS(R)$. Then $\Hom(A, \S^{2n}A) \iso 0$ for some $n\geq 0$ if and only if $A\iso 0$.
\end{cor}
\begin{proof}
The if direction is clear. For the other direction let $A$ be an object of $\SS(R)$ and $n>0$ (since the case $n=0$ is trivial) an integer such that $\Hom(A,\S^{2n}A) \iso 0$. Then for $1\leq i \leq c$ the map $A\to \S^{2n}A$ induced by $t_i^n$ is zero. Let $V_i$ denote the complement of $Y_{t_i}$. The localization $L_{V_i}$ sends $t_i^n$ to an isomorphism. We conclude that $L_{V_i}A$ is zero for each $i$ and so $A$ is zero by the local-to-global principle.
\end{proof}

\subsection{The refined filtration}
Although it is well known we begin with the observation that the maps in diagram~(\ref{eq_1}) give a functor $\GInj Y \to \GInj R$. We denote the right adjoint of $i_*$, which exists at the level of quasi-coherent sheaves as $i$ is a closed immersion, by $i^!$.

\begin{lem}\label{lem_ginj_functor}
Let $G$ be a Gorenstein injective quasi-coherent sheaf on $Y$. Then $p_*i^!G$ is a Gorenstein injective $R$-module. Thus $p_*i^!$ restricts to a functor
\begin{displaymath}
\psi\colon \GInj Y \to \GInj R.
\end{displaymath}
Moreover, this functor is exact and induces an equivalence $\underline{\psi}$ at the level of stable categories.
\end{lem}
\begin{proof}
Let $G$ be a Gorenstein injective quasi-coherent $\str_Y$-module and let $E$ be a complete injective resolution of $G$. By \cite{Stevensonclass}*{Lemma~8.5} the functor $i^!$ sends acyclic complexes of injectives to acyclic complexes of injectives. Thus $i^!E$ is an acyclic complex of injectives on $\mathbb{P}^{c-1}_R$ with zeroth syzygy $i^!G$ which shows that $i^!G$ is Gorenstein injective. Since truncating $E$ gives an injective resolution of $G$ this also shows that $G$ is acyclic for $i^!$ and so the functor $i^!$ is exact when restricted to $\GInj Y$.

An analogous argument shows that $p_*$ sends Gorenstein injectives to Gorenstein injectives and is exact when restricted to $\GInj \mathbb{P}_R^{c-1}$: since $p$ is flat $p_*$ sends injectives to injectives and by the proof of \cite{KrStab}*{Theorem~1.5} $p_*$ sends acyclic complexes of injectives to acyclic complexes.

As we have noted above $\psi$ sends injectives to injectives so induces a functor at the level of stable categories. That this is an equivalence is then just a different interpretation of the extension given in \cite{Stevensonclass}*{Proposition~8.7} of Orlov's equivalence.
\end{proof}

Suppose $G$ is a Gorenstein injective quasi-coherent sheaf on $Y$ and consider the filtration
\begin{displaymath}
0 = G_{\leq -1} \cie G_{\leq 0} \cie \cdots \cie G_{\leq \dim Y-1} \cie G_{\leq \dim Y} = G
\end{displaymath}
of Theorem~\ref{thm_gor_filt}. By the lemma $\psi$ is exact so we get a filtration of $\psi(G)$ by the $\psi(G_{\leq i})$ with factors
\begin{displaymath}
\psi(G_i) \iso \psi(G_{\leq i+1})/\psi(G_{\leq i}) \iso \psi (\coprod_{x\in \mathcal{H}_i^Y} G_x) \iso \coprod_{x\in \mathcal{H}^Y_i} \psi(G_x)
\end{displaymath}
where $\mathcal{H}^Y_i$ is the set of irreducible closed subsets of $Y$ of dimension $i$. The Gorenstein injective module $\psi(G_x)$, corresponding to a point $x$, is injective if and only if $G_x$ is injective. We thus have the following refined filtrations in $\GInj R$.

\begin{thm}\label{cor_ci_filt}
Let $H$ be a Gorenstein injective $R$-module. Then there is an injective $R$-module $E$ such that $H' = H\oplus E$ has a filtration
\begin{displaymath}
0 = H_{\leq -1}' \cie H_{\leq 0}' \cie \cdots \cie H_{\leq \dim Y-1}' \cie H_{\leq \dim Y}' = H'.
\end{displaymath}
Each factor $H_i'$ decomposes as a direct sum indexed by the points $x\in \mathcal{H}^Y_i$ and the modules $H_x'$ are 2-periodic up to injectives in the sense that
\begin{displaymath}
(\Sigma^2 H_x') \oplus I \iso H_x' \oplus J
\end{displaymath}
where $\Sigma$ denotes taking cosyzygies and $I$ and $J$ are injective $R$-modules.
\end{thm}
\begin{proof}
As explained in the preamble to the corollary the result essentially follows from Theorem~\ref{thm_gor_filt} by using Lemma~\ref{lem_ginj_functor}. Since $\underline{\psi}$ is an equivalence we can find an injective $E$ such that $H'= H\oplus E$ is in the essential image of $\psi$ and the existence of the claimed filtration follows from the previous discussion. The 2-periodicity, up to injectives, of the factors is a consequence of Proposition~\ref{prop_periodic}.
\end{proof}


\section{The associated spectral sequence}
The notation in this section is as in Sections \ref{sec_disc} and \ref{sec_filt}: $\mathcal{T}$ is a rigidly-compactly generated tensor triangulated category such that $\mathcal{T}$ is the homotopy category of some stable monoidal model category, $\Spc \mathcal{T}^c$ is noetherian, and we fix an action of $\mathcal{T}$ on some compactly generated triangulated category $\mathcal{K}$. We outline how to obtain local-to-global spectral sequences from the filtration associated to a dimension function on $\Spc \mathcal{T}^c$. Although the construction of the spectral sequence is entirely standard we believe it is worth making somewhat explicit. We then make some further remarks on interesting features of the spectral sequence for local complete intersections.

Let us fix some dimension function $\dim\colon \Spc \mathcal{T}^c \to \int\cup \{\pm \infty\}$. As in Section~\ref{sec_filt} this provides a filtration of $\mathcal{K}$
\begin{displaymath}
0 \cie \mathit{\Gamma}_{\leq -\infty}\mathcal{K} \cie \cdots \cie \mathit{\Gamma}_{\leq d-1}\mathcal{K} \cie \mathit{\Gamma}_{\leq d}\mathcal{K} \cie \cdots \cie \mathit{\Gamma}_{\leq \infty}\mathcal{K} = \mathcal{K}.
\end{displaymath}
and gives for each object $A$ of $\mathcal{K}$ a tower
\begin{displaymath}
\xymatrix{
\cdots \ar[r] &\mathit{\Gamma}_{\leq d-2}A \ar[r] & \mathit{\Gamma}_{\leq d-1}A \ar[r] \ar[d] & \mathit{\Gamma}_{\leq d}A \ar[r] \ar[d] & \cdots &\\
&& A_{d-1} \ar[ul]^{\S} & A_{d} \ar[ul]^{\S}
}
\end{displaymath}
Let $H\colon \mathcal{K} \to \mathcal{A}$ be a homological functor into an abelian category $\mathcal{A}$ having all small coproducts. As usual $H^i(-)$ denotes the composite $H\circ \S^i$ i.e., for $A\in \mathcal{K}$ we set $H^i(A) = H(\S^i A)$.

Now let us take an object $A$ of $\mathcal{K}$ and apply $H$ to its tower. We can organise the resulting information in the following commutative diagram (where the dashed arrows indicate an example of one of the long exact sequences corresponding to the triangles occuring in the tower for $A$)

\begin{displaymath}
\xymatrix{
H^{m-1}(\mathit{\Gamma}_{\leq n}A) \ar@{-->}[r]^-j \ar[d]_-i & H^{m-1}(A_n) \ar@{-->}[r]^-k & H^{m}(\mathit{\Gamma}_{\leq n-1}A) \ar@{-->}[d]_-i \ar[r]^-j & H^{m}(A_{n-1}) 
\\
H^{m-1}(\mathit{\Gamma}_{\leq n+1}A) \ar[r]^-j \ar[d]_-i & H^{m-1}(A_{n+1}) \ar[r]^-k & H^{m}(\mathit{\Gamma}_{\leq n}A) \ar[d]_-i \ar@{-->}[r]^-j & H^{m}(A_{n}) 
\\
H^{m-1}(\mathit{\Gamma}_{\leq n+2}A) \ar[r]^-j & H^{m-1}(A_{n+2}) \ar[r]^-k & H^{m}(\mathit{\Gamma}_{\leq n+1}A) \ar[r]^-j & H^{m}(A_{n+1}) 
\\
 }
\end{displaymath}

We now produce an exact couple from this diagram in the obvious way. We set
\begin{displaymath}
D(A,H)^{\bullet,\bullet} = \bigoplus_{p,q} H^{p+q}(\mathit{\Gamma}_{\leq q}A) \quad \text{and} \quad E(A,H)^{\bullet,\bullet} = \bigoplus_{p,q} = H^{p+q}(A_q)
\end{displaymath}
and consider the diagram
\begin{displaymath}
\xymatrix{
D(A,H) \ar[rr]^-i & & D(A,H) \ar[dl]^-j \\
& E(A,H) \ar[ul]^-k &
}
\end{displaymath}
where we will omit the $(A,H)$ if it is clear from the context. The morphisms here are induced in the obvious way from the maps $i,j,k$ occurring in the diagram organising the values of $H$ on the tower associated to $A$. It is easily verified that this is an exact couple with $\deg i = (1,-1)$, $\deg j = (0,0)$ and $\deg k = (2,-1)$. As usual we set $d = jk$ and can form the derived exact couple. Iterating this process gives a cohomological spectral sequence, starting on the second page, with $r$th differential $d_{r}$ of degree $(r,1-r)$.

The construction associating to each object $A$ of $\mathcal{K}$ the exact couple $(D,E)$ is functorial; this is a direct consequence of the functoriality of the towers.

The best case scenario is summarised in the following lemma which follows from the results we have proved in Section \ref{sec_filt} together with standard considerations concerning spectral sequences.

\begin{lem}
Suppose $\mathcal{K}$ has finite dimension with respect to the function $\dim$ in the sense that there are $m_{-},m_{+}\in \int$ with
\begin{displaymath}
\mathit{\Gamma}_{\leq m_{-}}\mathcal{K} = 0 \quad \text{and} \quad \mathit{\Gamma}_{\leq m_{+}}\mathcal{K} = \mathcal{K}.
\end{displaymath}
Then for any object $A$ of $\mathcal{K}$ and homological functor $H\colon \mathcal{K} \to \mathcal{A}$ the associated spectral sequence is zero outside of a horizontal strip between $m_{-}$ and $m_+$. In particular it converges to $H^*(A)$ i.e.
\begin{displaymath}
E_2^{p,q} = H^{p+q}(A_q) \rimp H^{p+q}(A).
\end{displaymath}
Furthermore, if $H$ preserves coproducts there are isomorphisms 
\begin{displaymath}
H^{p+q}(A_q) \iso H^{p+q}(\coprod_{x\in \mathcal{H}_q} \mathit{\Gamma}_x A) \iso \bigoplus_{x\in \mathcal{H}_q} H^{p+q}(\mathit{\Gamma}_x A).
\end{displaymath}
\end{lem}
Thus if $H$ is coproduct preserving the spectral sequences connects the local, with respect to the dimension function on $\Spc \mathcal{T}^c$, values of $H^*$ to the global value of $H^*$. This affords the possibility of applying the following trivial observation.

\begin{lem}\label{lem_hom_local}
Suppose $H$ is a coproduct preserving homological functor on $\mathcal{K}$ taking values in a cocomplete abelian category, let $\mathcal{L}$ be a minimal non-zero localizing subcategory of $\mathcal{K}$, and let $A$ be a non-zero object in $\mathcal{L}$. Then $H^i(A)$ vanishes for all $i\in \int$ if and only if for every $B\in \mathcal{L}$ and all $i\in \int$ the $H^i(B)$ vanish.
\end{lem}

\subsection{Local complete intersections}\label{sec_ss_ci}

We now make some comments on the spectral sequence in the special case that $(R,\mathfrak{m},k)$ is a local complete intersection of codimension $c$, $\mathcal{T} = \D(Y)$ where $Y$ is a generic hypersurface as in the discussion before Lemma \ref{lem_ginj_functor}, and $\mathcal{K} = \SS(R)$ (although one could work more generally). Here the local periodicity of $\SS(R)$ allows us to play certain games, using hypersurface sections, on the second page of the spectral sequence for a coproduct preserving homological functor.

For simplicity let us assume that $R$ has an isolated singularity i.e., for all non-closed $\mathfrak{p}\in \Spec R$ the ring $R_\mathfrak{p}$ is regular. We recall from \cite{Stevensonclass}*{Corollary~10.5} that the action of $\D(Y)$ on $\SS(R)$ gives a bijection between subsets of $\Sing Y = \mathbb{P}^{c-1}_k$ and localizing subcategories of $\SS(R)$. Let us fix some coproduct preserving homological functor $H$ on $\SS(R)$, which we may as well assume lands in the category of abelian groups, and an object $A\in \SS(R)$. Let $E$ be the associated spectral sequence for a dimension function assigning to a point of $\Sing Y$ the Krull dimension of its closure in $\Sing Y$ (we can obtain such a function by reindexing the Krull dimension function on $Y$). We have
\begin{displaymath}
E_2^{p,q} = 0 \;\; \text{for} \;\; q<0 \;\; \text{and} \;\; q>c-1
\end{displaymath}
so the $E_2$ page lives in a horizontal strip of width $c$.

We now consider the implications of Section \ref{ssec_periodicity}. We begin with the following easy consequence.

\begin{lem}\label{lem_ss_periodic}
The rows of the $E_2$ page are $2$-periodic i.e., there are isomorphisms
\begin{displaymath}
E_2^{p,q} \stackrel{\sim}{\to} E_2^{p+2,q}
\end{displaymath}
for all $p$ and $q$.
\end{lem}
\begin{proof}
In Proposition \ref{prop_periodic} we learned that, for each $y\in \Sing Y$, the subcategory $\mathit{\Gamma}_y\SS(R)$ is 2-periodic. Thus
\begin{displaymath}
A_q \iso \bigoplus_{\substack{y\in \mathbb{P}^{c-1}_k \\ \dim \mathcal{V}(y) = q}} \mathit{\Gamma}_yA \iso \bigoplus_{\substack{y\in \mathbb{P}^{c-1}_k \\ \dim \mathcal{V}(y) = q}} \S^2 \mathit{\Gamma}_yA \iso \S^2 A_q.
\end{displaymath}
It follows that $H^{p+q}(A_q) \iso H^{p+2+q}(A_q)$ as claimed.
\end{proof}

It is important to point out that these isomorphisms are \emph{not} natural if $c>1$ and that the differentials are \emph{not} in general periodic. However, we can still use this periodicity to move information about a given differential around; our goal is to formulate precisely what we mean by this. We will do this in terms of self-maps of the spectral sequence. As in Section \ref{ssec_periodicity} we denote by $\mathcal{L}$ the line bundle $u^*\str(1)_{\mathbb{P}^{c-1}_Q}$ on $Y$ and recall that it acts on $\SS(R)$ by $\S^2$.

\begin{lem}\label{lem_ss_map}
Let $s$ be a section of $\mathcal{L}^{\otimes n}$. Then $s$ induces a morphism of spectral sequences $\tilde{s}\colon E \to E$ of degree $(2n,0)$ via the natural morphisms
\begin{displaymath}
\xymatrix{
E(A,H)_2^{p,q} \ar[rr]^-{H^{p+q}(s)} && E(\mathcal{L}^{\otimes n} \otimes A,H)_2^{p,q} \iso E(A,H)_2^{p+2n,q}.
}
\end{displaymath}
\end{lem}
\begin{proof}
This is immediate: one tensors everything in sight with the morphism $\str_Y \to \mathcal{L}^{\otimes n}$ corresponding to $s$ and uses the last lemma.
\end{proof}



\begin{prop}
Suppose the residue field $k$ of $R$ is infinite. If the differential $d_{2}^{p,q}$ is injective then $d_{2}^{p-2n,q}$ is injective for all $n\geq 0$. Similarly if $d_{2}^{p,q}$ is surjective then $d_{2}^{p+2n,q}$ is surjective for all $n\geq 0$.
\end{prop}
\begin{proof}
Let us denote by $\mathcal{H}_q$ the set of points of $\Sing Y$ whose closure have dimension $q$. As $H$ is coproduct preserving we have for each $p,q$ isomorphisms
\begin{displaymath}
E_2^{p,q} = H^{p+q}(A_q) \iso H^{p+q}(\coprod_{y\in \mathcal{H}_q} \mathit{\Gamma}_yA) \iso \bigoplus_{y\in \mathcal{H}_q} H^{p+q}(\mathit{\Gamma}_yA).
\end{displaymath}
Suppose $d_{2}^{p,q}$ is injective and say $\a \in \ker d_{2}^{p-2n,q}$ where $n>0$. By the above direct sum decomposition we can find finitely many points $\{y_i\}_{i=1}^m$ in $\Sing Y$ such that $\a$ is zero outside of $\oplus_{i=1}^m H^{p+q-2n}(\mathit{\Gamma}_{y_i}A)$. We have assumed $k$ is infinite, so we can find a hypersurface in $Y$ of degree $n$, corresponding to a section $s$ of $\mathcal{L}^{\otimes n}$ missing all of the $y_i$. Lemma~\ref{lem_ss_map} then gives us a commutative square
\begin{displaymath}
\xymatrix{
E_2^{p-2n,q} \ar[d]_-{\tilde{s}} \ar[rr]^-{d_{2}^{p-2n,q}} && E_2^{p-2n+2),q-1} \ar[d]^-{\tilde{s}} \\
E_2^{p,q} \ar[rr]_-{d_{2}^{p,q}} && E_2^{p+2,q-1}
}
\end{displaymath}
So $d_{2}^{p,q}\tilde{s}(\a) = 0$. Since $\tilde{s}$ is an isomorphism when restricted to the components corresponding to the $y_i$ on which $\a$ is supported, as the section $s$ is invertible in a neighbourhood of each $y_i$, and $d_{2}^{p,q}$ is injective we deduce that $\a=0$. Thus $d_{2}^{p-2n,q}$ is injective as claimed.

The proof of surjectivity uses essentially the same ideas and so we omit the details.
\end{proof}

  \bibliography{greg_bib}

\end{document}